\documentclass[a4paper]{article}
\usepackage[english]{babel}
\usepackage[utf8]{inputenc}
\usepackage{amsmath,amsthm}
\usepackage{amsfonts,amssymb}

\usepackage{mathtools}

\usepackage[colorlinks=true,linktocpage=true,linkcolor=blue,citecolor=red]{hyperref}
\usepackage[capitalise]{cleveref}

\usepackage{tikz,tikz-cd}

\newcommand{\Ecal}{\mathcal{E}} 
\newcommand{\Tcal}{\mathcal{T}}  
\newcommand{\Scal}{\mathcal{S}} 
\newcommand{\Fcal}{\mathcal{F}}

\newcommand{\Pt}{\texttt{Pt}}
\newcommand{\topos}{\textsc{Topos}}
\newcommand{\Scott}{\mathbb{S}}

\theoremstyle{plain}

\newtheorem{theorem}{Theorem}[section]
\newtheorem{lemma}[theorem]{Lemma}
\newtheorem{prop}[theorem]{Proposition}
\newtheorem*{prop*}{Proposition}
\newtheorem{cor}[theorem]{Corollary}

\theoremstyle{definition}
\newtheorem{definition}[theorem]{Definition}

\theoremstyle{remark}
\newtheorem{npar}[theorem]{\unskip}

\begin{document}

\title{An abstract elementary class non-axiomatizable in $L_{\infty,\kappa}$.}
\author{Simon Henry}
\date{}
%\address{Masaryk university, Building 08 Kotl\'a\v{r}sk\'a 2, 611 37 Brno}

% \keywords{Toposes, points of toposes, Scott topology, Abstract elementary classes}
% \subjclass[2010]{18C35, 18C10, 03G30, 03C48}

%MSC classification:
% 18C35   	Accessible and locally presentable categories
% 18C10   	Theories (e.g. algebraic theories), structure, and semantics
% 03G30   	Categorical logic, topoi
% 03C48   	Abstract elementary classes and related topics

\maketitle

\begin{abstract}
We show that for any uncountable cardinal $\lambda$, the category of sets of cardinality at least $\lambda$ and monomorphisms between them cannot appear as the category of points of a topos, in particular is not the category of models of a $L_{\infty,\omega}$-theory. More generally we show that for any regular cardinal $\kappa < \lambda$ it is neither the category of $\kappa$-points of a $\kappa$-topos, in particular, not the category of models of a $L_{\infty,\kappa}$-theory.

The proof relies on the construction of a categorified version of the Scott topology, which constitute a left adjoint to the functor sending any topos to its category of points and the computation of this left adjoint evaluated on the category of sets of cardinality at least $\lambda$ and monomorphisms between them. The same techniques also applies to a few other categories. At least to the category of vector spaces of with bounded below dimension and the category of algebraic closed fields of fixed characteristic with bounded below transcendence degree.

\end{abstract}

%\tableofcontents

\renewcommand{\thefootnote}{\fnsymbol{footnote}} 
\footnotetext{This work was supported by the Operational Programme Research, Development and Education Project ``Postdoc@MUNI'' (No. CZ.02.2.69/0.0/0.0/16\_027/0008360)}
%\footnotetext{\emph{email:} simon.henry@college-de-france.fr}
\renewcommand{\thefootnote}{\arabic{footnote}} 

\section{Introduction}

Determining which categories can be obtained as categories of points of a topos is in general a difficult question. In this paper we answer a question of Ji\v{r}\'i Rosick\'y, that he asked during his talk ``Towards categorical model theory'' at the 2014 category theory conference in Cambridge, and also mentioned in section 5 of \cite{beke2012abstract}:

\bigskip

\emph{Show that the category of uncountable sets and monomorphisms between cannot be obtained as the category of point of a topos. Or give an example of an abstract elementary class that does not arise as the category points of a topos.}

\bigskip

We will prove some more general claims along the same lines (~\cref{Cor:MainFinitary,Th:SetLambdaNotKappapoints}).

This question of whether a given category is the category of points of a topos should be of interest to model theorist for the following reason:

\begin{theorem}
A category is the category of points of a topos if and only if it is the category of models of geometric theory and morphisms of structures between them.
\end{theorem}

This is (a relatively weak form of) the very famous theorem that every geometric theory admits a classifying topos and that every topos is the classifying topos of such a theory. We refer the reader to any good topos theory book for a more detailed introduction to these ideas, for example \cite{maclane1992sheaves}, \cite{sketches} or \cite{borceux3}. 

A geometric theory, in a signature $\Sigma$ (including sort, functions and $n$-ary relations) is a theory whose axiom can be written in the form:
\[ \forall x_1,\dots,x_n, (P(x_1,\dots,x_n) \Rightarrow Q(x_1,\dots,x_n)) \]
where $P$ and $Q$ are ``geometric proposition'' i.e. they are  built from:

\begin{itemize}

\item True and False.

\item Atomic formulas, i.e. equality or relation applied to terms (which are either variables, or function in $\Sigma$ applied to variables)

\item Finite conjunction (``and'').

\item arbitrary disjunction (``or'').

\item existential quantification $\exists$.

\end{itemize}

The asymmetry between disjunction and conjunction and $\exists/ \forall$ in these theories comes from the fact that they are made to be studied in an intuitionist framework. But even in purely classical framework it plays a very important role in the theory.

Given a theory using axioms in $L_{\infty,\omega}$, i.e. axioms allowing both arbitrary (infinite) conjunction and disjonction. The process of ``\emph{Morleyisation}'' allows to turn it into a geometric theory, at the price of changing the signature (see for example D1.5.13 of \cite{sketches} for an explicit description of this construction for finitary first order logic, a description of a similar construction for $L_{\kappa^+,\kappa}$ can also be found in \cite{espindola2017infinitary}, this is also briefly alluded in section 3.4 of \cite{makkai1989accessible}).

Essentially, it consists in adding to the signature two new relation symbols for each relation definable in $L_{\infty,\omega}$ corresponding respectively to the proposition and its negation, as well as new axioms that only forces these new relations to be what they are meant to be.

As this requires changing the signature, it generally changes the morphisms. If we follow the ``maximal'' Morleyisation described above and add new relations for each definable proposition, this makes the morphisms exactly the elementary embeddings of the $L_{\infty,\omega}$-theory. Though depending on the theory, it might be possible to take only partial Morleyisation, i.e. only adding the symbol that are needed and to get other notion of morphisms.

As a consequence of this, a special case of J.Rosick\'y question is to show that the category of uncountable sets cannot be axiomatised in $L_{\infty,\omega}$. We will show more generally (using the notion of $\kappa$-topos) that:

\begin{theorem}
\label{Th:SetLamndaNotKappamodel}For any cardinal $\lambda$ and regular cardinal $\kappa < \lambda$, the category $Set^m_{\geqslant \lambda}$ of sets of cardinality at least $\lambda$ and monomorphisms between them is not equivalent to the category of models and elementary embeddings of a theory axiomatizable in $L_{\infty,\kappa}$.
\end{theorem}

The argument can also be adapted to show similarly that the category of $K$-vector spaces of dimension at least $\lambda$ and linear monomorphisms between them, as well as the category of algebraic closed fields of fixed characteristic that are of transcendence degree at least $\lambda$ over their prime field, cannot be obtained as models of such theories as well. This will be briefly discussed in \cref{sec:generalizations}.

\section[The Scott topos construction (Joint work with Ivan Di Liberti)]{The Scott topos construction (Joint work with Ivan Di Liberti\protect\footnote{The construction of the adjunction $\Scott \dashv \Pt$ presented in this section and some of its property comes from an unpublished joint work with Ivan Di Liberti. As he did not took part in the proof of the main result of this paper (\cref{Prop:SbSetLambdaConstant,Th:SetLambdaNotKappapoints}) he decided to not be included as an author of the present paper. His contribution on the topic will appear in his PhD thesis.})}
\label{section:finitaryScott}

\begin{npar}
The category of points of a topos is not arbitrary. If $\Tcal$ is a topos, then its category of points $\Pt(\Tcal)$ is always an accessible\footnote{We refer to \cite{adamek1994locally} for the general theory of accessible categories.} category, moreover it has all filtered colimits. Indeed, a point $p$ of $\Tcal$ is a functor:
\[ \Tcal \overset{p^*}{\rightarrow} Set \]
which preserve arbitrary colimits and finite limits, we take morphisms of points to be just natural transformations between such functors.

As filtered colimits commute to both finite limits and arbitrary colimits, a filtered colimits of points (in the category of functors) is again a point. In particular, if $f : \Tcal \rightarrow \Ecal$ is any geometric morphisms then the induced functor $\Pt(\Tcal) \overset{f}{\rightarrow} \Pt(\Ecal)$ preserves filtered colimits. This induces a functor:
\[ \Pt : \topos \rightarrow Acc_{\omega} \]
where $Acc_{\omega}$ is the category of accessible category admitting filtered colimits and functor preserving filtered colimits between them.
\end{npar}

\begin{npar}
We can construct a left adjoint to $\Pt$, that we denote $\Scott$. It is defined for $A \in Acc_{\omega}$ as:
\[ \Scott A = \{ F: A \rightarrow Set | F \text{ preserves filtered colimits } \} \]
morphisms being just the natural transformation.
\end{npar}

\begin{prop*}$\Scott A$ is a Grothendieck topos. \end{prop*}

\begin{proof}
As finite limits and arbitrary colimits both commute to filtered colimits, $\Scott A$ has finite limits and arbitrary colimits, which are computed objectwise (i.e. in the category of all functors). In particular it clearly satisfies all of Giraud's limits/colimits axioms: coproduct are disjoint and pullback stable, congruences are effective with pullback stable quotient, as those can all be checked ``objectwise''. So it only remains to check that $\Scott A$ is an accessible category. 

If we fix some cardinal $\lambda$ such that $A$ is $\lambda$-accessible, and with $A_{\lambda}$ the category of $\lambda$-presentable objects of $A$ one has an equivalence:
\[ \Scott A \simeq \{ F: A_{\lambda} \rightarrow Set | F \text{ preserves $\lambda$-small $\omega$-filtered colimits}\} \]
which shows that $\Scott A$ is sketchable (i.e. it is written as a category of functors on a small category preserving to some colimits, see \cite[sec 2.F]{adamek1994locally}) hence accessible by \cite[Cor. 2.61]{adamek1994locally}. To prove the equivalence above, one observes that $A = Ind_{\lambda} A_{\lambda}$, hence the category of functors from $A$ to Sets preserving $\lambda$-directed colimits is equivalent to the category of all functors from $A_{\lambda}$ to Sets (by Kan extension in one direction and restriction in the other). 

Through that equivalence, the functor from $A_{\lambda}$ to Sets preserving $\lambda$-small $\omega$-filtered colimits corresponds exactly to the functor from $A$ to sets preserving $\lambda$-directed colimits and all $\lambda$-small $\omega$-filtered colimits of $\lambda$-presentable objects. But this second condition is easily seen to be equivalent to the preservation of all $\omega$-directed colimit.
\end{proof}

\begin{prop}
The functor $\Scott$ defines a left adjoint to $\Pt$:
\[ \Scott : Acc_{\omega} \rightleftarrows \topos : \Pt.\]
\end{prop}

\begin{proof}
 Let $\Tcal$ be a topos, and $X \in \Tcal$ be any object. Then $p \mapsto p^* X$ gives a filtered colimit preserving functor $\Pt(\Tcal) \rightarrow Set$ This defines a functor $\Tcal \rightarrow \Scott (\Pt(\Tcal))$ which clearly commutes to all colimits and finite limits. Hence it corresponds to a geometric morphisms $\Scott (\Pt(\Tcal)) \rightarrow \Tcal$.  Given $A$ an accessible category with all filtered colimits, and $a \in A$ an object, the functor of evaluation at $a$ gives a points of the topos $\Scott A$, and this is produces a functor $A \rightarrow \Pt( \Scott A)$ preserving all filtered colimits. It is then easy to check that these two functors are natural in $A$ and $\Tcal$ and satisfy the usual relations to be the co-unit and unit of an adjunction between $\Scott $ and $\Tcal$.
\end{proof}

As suggested by the title of this section, the letter ``$\Scott $'' is for ``Scott'', and we refer to ``$\Scott A$'' as the ``Scott topos of $A$''. The reason is that this construction is a categorification of the usual Scott topology on a directed complete poset: 

\begin{definition}
Given a poset $P$ with directed suprema, a Scott open subsets is a subset  $U \subset P$ such that if $u = \bigcup u_i$ is a directed supremum and $u \in U$, then $\exists i , u_i \in I$.
\end{definition}

Scott open subsets forms a topology on $P$, called the Scott topology on $P$. The definition of a Scott open can be rewritten as a function:
\[ P \rightarrow \Omega = \{ \bot, \top \} \]
which is non-decreasing and preserves directed supremums. A poset with all directed supremums is in particular an accessible category with all filtered colimits, and this description of Scott open subsets identifies them with the subterminal objects of our Scott topos $\Scott P$. Indeed, the terminal object of $\Scott P$ is the functor $P \rightarrow Sets$ constant equal to $\{*\}$, so a subobject is a functor sending each $p \in P$ to either $\{*\}$ or $\emptyset$ and preserving directed colimits, i.e. a non-decreasing function $P\rightarrow \Omega$ preserving directed suprema:

\begin{prop}
The frame of Scott open of a poset $P$ with directed suprema is the localic reflection of the Scott topos $\Scott P$.
\end{prop}

It is not clear to us if $\Scott P$ is in general a localic topos when $P$ is a poset, i.e. if the geometric morphism $\Scott P \rightarrow Sh(\Scal P)$ is always an equivalence. In practice it seems to be quite often an equivalence, but it also seems unlikely to be true in general.

The general properties of this adjunction are still unclear at this point, so I will not discuss them further. This might be the object of a future work. I'll just mention that the unit of adjunction $A \rightarrow \Pt \Scott A$ is not always faithful:

\begin{prop}
The functor $A \rightarrow \Pt \Scott A$ is faithful if and only if $A$ admits a faithful functor to the category of sets which preserves filtered colimits.
\end{prop}

See example $4.14$ in \cite{beke2012abstract} for an example of an accessible category with directed colimits with no such faithful functor to the category of sets, and hence for which the unit of adjunction is not faithful.

\begin{proof}
If $A \rightarrow \Pt \Scott A$ is faithful then evaluation at a bound\footnote{For example the coproduct of representable sheaves for any small site of definition of the topos} of the topos $\Scott A$ produces such a faithful functor, and conversely if there is such a faithful functor, then it gives a single object of the topos $\Scott A$ such that evaluation at this object induces a faithful composite $A \rightarrow \Pt \Scott A \rightarrow Set $ functor, which shows that $A \rightarrow \Pt \Scott A$ is faithful.
\end{proof}

\begin{npar}
Given an accessible category $A$ with directed colimits, a good place to start if one wants to know whether $A$ is the category of points of a topos is to compute $\Scott A$. Indeed the adjunction $\Scott \dashv \Pt$, means that the data of a functor $F:A \rightarrow \Pt (\Tcal)$ preserving filtered colimits is the same as a geometric morphism $F':\Scott(A) \rightarrow \Tcal$. Moreover one recover $F$ from $F'$ as the composite:
 \[ A \rightarrow \Pt(\Scott(A)) \overset{\Pt(F)}{\rightarrow} \Pt(\Tcal). \]
In practice, it seems to happen quite often that $A \simeq Pt \Scott A$ in which case the problem is solved. If $A \rightarrow \Pt \Scott A$ is not an equivalence, it might still be the case, as far as we know, that $A$ is the category of point of a different topos. But this impose serious restriction: for example if $A = \Pt \Tcal$, it must be the case that $A$ is a retract of $\Pt \Scott A$ with the retraction preserving directed colimits. This can either give hints on what $\Tcal$ should be or produce a proof that $A$ is not the category of point of a topos. Our main result in this paper is an application of this idea:

\end{npar}

Let $\lambda$ be any cardinal, and let $Set^m_{\geqslant \lambda}$ be the category of sets of cardinality at least $\lambda$ and monomorphisms between them. In \cref{sec:computing}, we will show that:

\begin{prop}
\label{Prop:SbSetLambdaConstant}For any $\lambda \geqslant \omega$ the inclusion $Set^m_{\geqslant \lambda} \subset Set^m_{\geqslant \omega}$ induces an equivalence:
\[ \Scott Set^m_{\geqslant \lambda} \simeq \Scott Set^m_{\geqslant \omega} \]
\end{prop}

Moreover, we will see that $\Scott Set^m_{\geqslant \omega}$ identifies with the Schanuel topos $Sh(Set^m_{< \omega},J_{at})$ (see \cref{Cor:Sec3main}). It is well known that the Schanuel topos is the classifying topos for the theory of infinite decidable sets, so in this case the canonical map:
\[ Set^m_{\geqslant \omega} \overset{\simeq}{\rightarrow} \Pt \Scott Set^m_{\geqslant \omega} \]
is an equivalence.

\begin{cor}
\label{Cor:MainFinitary}For any uncountable cardinal $\lambda$, the abstract elementary class $Set^m_{\geqslant \lambda}$ is not equivalent to the category of points of a topos, in particular it is not equivalent to the category of models and elementary embeddings of a theory axiomatizable in $L_{\infty,\omega}$.
\end{cor}

\begin{proof}
We mentioned before that the category of models of a $L_{\infty,\omega}$-theory and elementary embeddings between them is the category of points of the classifying topos of the Morleyisation of the theory.

If $Set^m_{\geqslant \lambda} \simeq \Pt(\Tcal)$ for some topos $\Tcal$, then this isomorphism is adjoint to a morphism:
\[ \Scott (Set^m_{\geqslant \lambda}) \rightarrow \Tcal \]
but then the isomorphism $ \Scott (Set^m_{\geqslant \omega}) \simeq \Scott (Set^m_{\geqslant \lambda}) \rightarrow \Tcal $ corresponds through the adjunction $\Scott \dashv \Pt$ to functor:
\[ C: Set^m_{\geqslant \omega} \rightarrow \Pt(\Tcal) \simeq Set^m_{\geqslant \lambda} \]
which fits into the commutative diagram:
\[\begin{tikzcd}[ampersand replacement=\&]
Set^m_{\geqslant \lambda} \arrow[bend left=30]{rr}{\simeq} \arrow{r} \arrow{d} \& \Pt \Scott (Set^m_{\geqslant \lambda}) \arrow{d}{\simeq} \arrow{r} \& \Pt \Tcal \\
Set^m_{\geqslant \omega} \arrow{r}{\simeq} \arrow[bend right=60]{urr}[swap]{C} \& \Pt \Scott (Set^m_{\geqslant \omega})
\end{tikzcd}\]
where the upper curved arrow is our chosen identification $Set^m_{\geqslant \lambda} \simeq \Pt(\Tcal)$ and the square on the left is just the naturality of the co-unit of adjunction $Id \rightarrow \Pt \Scott$ applied to the inclusion $Set^m_{\geqslant \lambda} \subset Set^m_{\geqslant \omega}$. In particular, $C$ restricted to $Set^m_{\geqslant \lambda} \subset Set^m_{\geqslant \omega}$ is equivalent to the identity functor on $Set^m_{\geqslant \lambda}$ (through the identification $\Pt \Tcal \simeq Set^m_{\geqslant \lambda}$).

This yields a contradiction: consider $A,B \subset \lambda$ two subset of cardinality $\lambda$ such that $A \cap B$ is of cardinality $\omega$. By functoriality of $C$ one has a commutative diagrams:
\[\begin{tikzcd}[ampersand replacement=\&]
C( A\cap B) \arrow{d} \arrow{r} \& C(A) \arrow{d} \\
C(B) \arrow{r}\& C(\lambda)
\end{tikzcd}\]
in $Set^m_{\geqslant \lambda}$. As $A,B, \lambda$ are of cardinality $\lambda$, $C$ is equivalent to the identity on them (and on the maps between them), hence the intersection $C(A) \cap C(B)$ in $C(\lambda)$ is isomorphic to the intersection of $A$ and $B$ hence is countable. Hence the diagram above induce a map $C(A \cap B) \rightarrow C(A) \cap C(B)$ which is a monomorphism from a set of cardinality at least $\lambda$ to a countable set, which is our contradiction.
\end{proof}

\section{The Scott $\kappa$-topos}
\label{section:infinitaryScott}

\begin{npar}
We would now like to extend \cref{Cor:MainFinitary} to prove \cref{Th:SetLamndaNotKappamodel}. In order to achieve that, we will use the machinery of $\kappa$-geometric toposes developed by C.Espindola in \cite{espindola2017infinitary}. In what follows, $\kappa$ will always denotes a regular cardinal.
\end{npar}

\begin{definition}
A $\kappa$-topos is a $\kappa$-exact localization of a presheaf category. i.e. a reflective subcategory of a presheaf category such that the reflection preserve all limits indexed by diagram of size $<\kappa$ (latter called $\kappa$-small limits)
\end{definition}

In \cite{espindola2017infinitary}, C.Espindola consider a class of toposes he called ``$\kappa$-geometric toposes'' that are defined as the Grothendieck toposes satisfying a further exactness property called property ``$T$'' (see \cite{espindola2017infinitary} definition 2.1.1 and the paragraph above it) which can be summarized as: ``a $\kappa$-small transfinite composition of covering families is a covering family''. He shows in the proof of his theorem $3.0.1$, that if a site with $\kappa$-small limits satisfies this property $T$, then the corresponding localization is $\kappa$-exact. By applying this to any site of definition with $\kappa$-small limits of a $\kappa$-geometric topos this shows that any $\kappa$-geometric topos in his sense is indeed a $\kappa$-topos in ours. It should be noted however that, contrary to what was claimd in an earlier version of this paper, his condition is stronger than being a $\kappa$-exact localization.

\bigskip

A $\kappa$-geometric morphism  between $\kappa$-toposes $\Tcal \rightarrow \Ecal$, is a functor $\Ecal \rightarrow \Tcal$ which preserves all colimits and $\kappa$-limits. In particular, a $\kappa$-point of a topos is a functor $\Tcal \rightarrow Sets$ which preserves all colimits and $\kappa$-limits.

\bigskip

The connection between points of toposes and models of $L_{\infty,\omega}$-theories has been generalized by C.Espindola in \cite{espindola2017infinitary} to the similar connection between $\kappa$-points of $\kappa$-geometric toposes and models of $L_{\infty,\kappa}$-theories (C.Espindola described Morleyisation for $L_{\kappa^+,\kappa}$, but everything generalizes immediately to $L_{\infty,\kappa}$) . So \cref{Th:SetLamndaNotKappamodel} above follows from:

\begin{theorem}
\label{Th:SetLambdaNotKappapoints}For any cardinal $\kappa < \lambda$ the category $Set^m_{ \geqslant \lambda}$ is not equivalent to the category of $\kappa$-points of a $\kappa$-topos.
\end{theorem}

In order to prove \cref{Th:SetLambdaNotKappapoints}, we will follow the same strategy as for the proof of \cref{Cor:MainFinitary}, appropriately generalized to this context of $\kappa$-toposes. The following claims have the same proofs\footnote{For point $3.$ one checks that this topos satisfies Espindola condition $T$: the condition only involves colimits and $\kappa$-small limits and epimorphisms, which in $\Scott_{\kappa} A$ are computed/detected levelwise, so it follows from the fact that the condition holds in Sets.} as their finitary counterpart proved in \cref{section:finitaryScott}. 

\begin{prop}
\begin{enumerate}
\item[]
\item If $\Tcal$ is a $\kappa$-topos, the category $\Pt_{\kappa}(\Tcal)$ of $\kappa$-points of $\Tcal$ is an accessible categories with $\kappa$-filtered colimits.

\item $\Pt_{\kappa}$ defines a functor from the category $\kappa\text{-}\topos$ of $\kappa$-toposes and $\kappa$-geometric morphisms to the category of $Acc_{\kappa}$ of accessible categories with functor preserving $\kappa$-filtered colimits between them.

\item Given an accessible category $A$ with $\kappa$-directed colimits the category:
\[ \Scott_{\kappa} A = \{ F: A \rightarrow Set | F \text{ preserves $\kappa$-filtered colimits } \} \]
is a $\kappa$-geometric topos (in particular a $\kappa$-topos).

\item One has an adjunction $\Scott_{\kappa} \dashv \Pt_{\kappa}$ :
\[ \Scott_{\kappa} : Acc_{\kappa} \leftrightarrows \kappa\text{-}\topos : \Pt_{\kappa}. \]
\end{enumerate}

\end{prop}

In \cref{Cor:Sec3main} we will prove more generally:

\begin{prop}
\label{Prop:SbkappaEquivOnSetLambda}For any $\kappa \leqslant \lambda$, the inclusion $Set^m_{\lambda} \subset Set^m_{\kappa}$ induces an equivalence of categories:
\[ \Scott_{\kappa} Set^m_{\lambda} \overset{\simeq}{\rightarrow} \Scott_{\kappa} Set^m_{\kappa}. \]
\end{prop}

The proof of \cref{Cor:MainFinitary} then proves in the same way that this proposition implies \cref{Th:SetLambdaNotKappapoints} and hence \cref{Th:SetLamndaNotKappamodel} as well.

\section{Computing  $\Scott_{\kappa} Set^m_{\geqslant \lambda}$.}
\label{sec:computing}

As the title suggest the goal of this section is to understand the topos $\Scott_{\kappa} Set^m_{\geqslant \lambda}$ for any two fixed cardinal $\lambda \geqslant \kappa$ with $\kappa$ regular. More precisely, we want to prove \cref{Prop:SbkappaEquivOnSetLambda}, and \cref{Prop:SbSetLambdaConstant} which is essentially the special case $\kappa = \omega$.

\begin{npar}
We start by introducing some objects of $\Scott_{\kappa} Set^m_{\geqslant \lambda}$:

For any set $V$ of cardinality strictly smaller than $\kappa$, the functor:
\[R_V: \begin{array}{c c l}
Set^m_{\geqslant \lambda} & \rightarrow & Set \\
S & \mapsto & \{ \text{Monomorphisms } V \rightarrow S \}
\end{array}\]

Is an element of $\Scott_{\kappa} Set^m_{\geqslant \lambda}$. Moreover this construction $V \mapsto R_V$ naturally defines a functor: 
\[R_{\bullet} : (Set^m_{< \kappa})^{op} \rightarrow \Scott_{\kappa} Set^m_{\geqslant \lambda} \]
where $Set^m_{< \kappa}$ is the category of sets of cardinality smaller than $\kappa$ and monomorphisms between them. Our main result in this section is that the natural ``Nerve'' functor:
\[ \Scott_{\kappa} Set^m_{\geqslant \lambda} \rightarrow Prsh\left( (Set^m_{< \kappa})^{op} \right)  \]
induced by $R$ will identify $\Scott_{\kappa} Set^m_{\geqslant \lambda}$ with the category of sheaves on $(Set^m_{< \kappa})^{op}$ for the atomic topology (i.e. the topology where every non-empty sieve is a cover). In order to prove that one needs to better understand morphisms between the $R_V$, and more generally morphisms from $R_V$ to any other objects.

\end{npar}

We fix $\lambda$ an infinite cardinal and $F \in \Scott_{\kappa} (Set^m_{\geqslant \lambda})$, i.e. a $\kappa$-filtered colimits preserving functor $Set^m_{\geqslant \lambda} \rightarrow Set$.

\begin{definition}
An element $x \in F(X)$ is said to have support in $K \subset X$ if for all $f,g: X \rightrightarrows Y$ one has:
 \[ f_{|K} = g_{|K} \Rightarrow f x = g x \]
\end{definition}

For example, if $v:Y \rightarrow X$ is a map in $Set^m_{\geqslant \lambda}$ and then for $y \in F(Y)$ the element $vy \in F(X)$ has support in $Y$. But the converse doesn't have to be true, and contrary to this observation the notion of ``support in $Y$'' makes sense for $Y$ of cardinality smaller than $\lambda$.

Here $f x, g x$ or $v y$ of course means $F(f) (x)$,$F(g) (x)$ and $F(v)(y)$. This abuse of notation will be use constantly in the text.

\begin{lemma}
\label{lemma:morphismOutOfR_V}Fix a monomorphism $\tau:V \hookrightarrow \lambda$, i.e. $ \tau \in R_V(\lambda)$.

Then for $V \in Set^m_{< \kappa}$ and $\Fcal \in \Scott_{\kappa} Set^m_{\geqslant \lambda} $  one has a bijection:
\[ Hom(R_V, \Fcal) \overset{\simeq}{\rightarrow}  \{ x \in\Fcal(\lambda) | x \text{ has support in $\tau(V)$} \}\]
sending a morphism $f :R_V \rightarrow \Fcal$ in $\Scott_{\kappa} Set^m_{\geqslant \lambda}$  to $f(\tau) \in \Fcal(\lambda)$.

\end{lemma}

Note that this is basically the Yoneda lemma, and the proof is essentially the usual proof of the Yoneda lemma, where ``$\tau$'' and the notion of support have been added to fix the problem that the $V \in Set^m_{< \kappa}$ are not part of the category on which $\Fcal$ is defined.

\begin{proof}

 One easily see that $\tau \in R_V(\lambda)$ has support in $\tau(V)$, and this implies that given a morphism $f :R_V \rightarrow \Fcal$ the image $f(\tau)$ has support in $\tau(V)$ as well. i.e. the morphism in the lemma is well defined.

Conversely, given any $x \in \Fcal(\lambda)$ with support in $\tau(V)$ and given $w \in R_V(Y)$, $w$ is an injective map from $V$ to $Y$, for any extension $\widetilde{w}$ of $w$ to an injective map $\lambda \hookrightarrow Y$, the value of $\widetilde{w} x$ only depends on $w$ as $x$ as support in $V$, so we define $w x := \widetilde{w} x$ for any such extension $\widetilde{w}$. One easily see that $w \mapsto w x$ is a natural transformation from $R_V$ to $\Fcal$ which sends $\tau$ to $x$. Conversely any morphism $f:R_V \rightarrow \Fcal$ sending $\tau$ to $x$ is equal to this one: indeed given $w \in R_V(Y)$ and $\widetilde{w}$ an extension of $w$ along $\tau$ as a morphism $\lambda \rightarrow Y$ then $w = \widetilde{w}.\tau$ so by naturality of $f$ one have $f w = \widetilde{w} (f \tau ) = \widetilde{w} x$.
\end{proof}

In the rest of this section one will say that a set $X$ is $\kappa$-small if the cardinal of $X$ is strictly less than $\kappa$. In the case $\kappa =\omega$ this just means finite.

\begin{prop}
\label{finitesupport}For any $\Fcal \in \Scott_{\kappa} Set^m_{\geqslant \lambda}$, any element of $\Fcal(X)$ admits a $\kappa$-small support. I.e. it has support in a set $V \subset X$ of cardinality strictly less than $\kappa$.
\end{prop}

The general idea of the proof is as follows. If $\kappa$-small sets were available in our category, we could just say that:
\[ X = \underset{K \subset X \atop K \text{ $\kappa$-small}}{colim} K \]
Hence as $F$ commutes with $\kappa$-filtered colimits any elements of $F(X)$ should be in the image of $F(K)$ for some $\kappa$-small $K$ and hence have support in $K$. Of course this is not possible as $F(K)$ is not defined for $\kappa$-small $K$. Our proof will rely instead on the following filtered colimits:

\begin{equation}\label{colim1} X \coprod X = \underset{K \subset X \atop K \text{ $\kappa$-small}}{ colim} \left( X \coprod K \right) \end{equation}

And a few tricks organized in the following two lemmas.

\begin{lemma}
\label{Lemma1}\begin{itemize}
\item[]

\item Let $x \in F(X \coprod X)$ then $x$ has support in $X \coprod K$ for some $\kappa$-small $K \subset X$.

\item Let $x \in F(X)$ then there exists an isomorphism $\theta: X \overset{\sim}{\rightarrow} X \coprod X$ such that $\theta x \in F(X \coprod X)$ has support in the second component $i_2:X \hookrightarrow X \coprod X$. 

\end{itemize}

\end{lemma}

\begin{proof}

\begin{itemize}

\item[]

\item This follows immediately from the colimits in (\ref{colim1}).

\item As $X$ is infinite, one can find an isomorphism $\theta_0 :X \simeq X \coprod X$, then because of the colimit (\ref{colim1}), $\theta_0 x $ has support in $X \coprod K$ for some $\kappa$-small subset $K \subset X$. As $X$ has cardinal at least $\lambda$, and $K$ strictly less than $\kappa \leqslant \lambda$, one can then find some automorphisms $\theta_1$ of $X \coprod X$ that send $X \coprod K$ to the the second coproduct inclusion $i_2(X) \subset X \coprod X$ the composite $\theta_1 \theta_0$ has the property required by the lemma.
\end{itemize}
\end{proof}

\begin{lemma}
\label{Lemma2}Let $x \in F(X)$ which has support both in $A \subset X$ and $B \subset X$ such that $X =A \cup B$ and $K = A \cap B$ is $\kappa$-small. Then $x$ has support in $K$.
\end{lemma}

The fact that ``supports intersect'' seems to be true much more generally, but the proof of this would involve a painful case distinction on the various size of $A$,$B$ there intersection and the complement of the union. so I decided to focus on the case which is useful to us. I haven't really checked if all the cases in the more general situation work, but this is strongly expected.

\begin{proof}

Let $A' = A-K$ and $B'=B-K$ so that $X = A' \coprod B' \coprod K$. We start with two monomorphisms $f,g:X \rightrightarrows Y$ such that $f_{|K} = g_{|K}$, we will gradually replace $f$ and $g$ by new monomorphisms $X \rightarrow Y$ which agree with the previous ones either on $A$ or on $B$ (all called $f$ or $g$ for simplicity) ,hence, such that the value of $fx$ and $gx$ remains unchanged, and at the end we will have $f = g$. This will prove that $fx = gx$ for the original $f$ and $g$ and hence that $x$ indeed has support in $K$ as claimed.

Let $\mu$ be the cardinality of $X$.
First, one modifies $f$ and $g$ to make sure that both avoid some subsets $S_f$ and $S_g$ of $Y$ of cardinality $\mu$. This can be done by modifying $f$ and $g$ only on $A'$, by shrinking the image $f(A')$ to a subset $V \subset f(A')$ such that $|V|=|f(A') - V| =\mu$ and letting $S_f = f(A')-V$, and similarly for $g$.

If $S_f \cap S_g$ is of cardinality $\mu$, then we redefine $S_f = S_g = S_f \cap S_g$. If $S_f \cap S_g$ is of cardinality strictly less than $\mu$ then we redefine $S_f := S_f -(S_f \cap S_g)$ and $S_g := S_g - (S_f \cap S_g)$. In both case $f$ and $g$ avoids respectively $S_f$ and $S_g$ both of cardinality $\mu$ and either $S_f = S_g$ or $S_f \cap S_g =\emptyset$. In the first case one can fix a monomorphism $i:A' \coprod B' \rightarrow S_f = S_g$ and make both (the restriction of) $f$ and $g$ equal to it in two steps for each: one modifies $f$ on $A'$ to make it equal to $i_{|A'}$, as $f$ avoids $S_f$ the resulting maps is still a mono, and then one modifies it similarly on $B'$, and one do the same for $g$. At this points $f=g$ on $A'$, $B'$ and $K$, so $f=g$ and the proof is done.

In the case $S_f \cap S_g = \emptyset$ , one fixes a monomorphism $i:S_g \rightarrow S_f$ and one modifies $f$ so that it avoids $S_g$ instead of $S_f$: first all the elements of $A'$ that were sent to $S_g$ are sent to their image by $i$ instead, and then one do the same for the elements of $B'$ in a second step. At this point $f$ and $g$ both avoid $S_g$ and we are brought back to the previous case.
\end{proof}

\begin{npar}
One can now easily prove \cref{finitesupport}: Starting from $x \in F(X)$ then, because of the second point of \cref{Lemma1}, it is isomorphic to a $x' \in F(X \coprod X)$ by a $\theta:X \simeq X \coprod X$ and $x'$ is supported on the second component. By the first point of \cref{Lemma1} one also has that $x'$ is supported in $X \coprod K$ for some $K$ finite. Hence by \cref{Lemma2} it is supported in $K$. It follows that $x$ has support in $\theta^{-1} K$ which is also finite.
\end{npar}

\begin{cor}
\label{Cor:Sec3main}The functor:
\[R_{\bullet} : (Set^m_{< \kappa})^{op} \rightarrow \Scott_{\kappa} Set^m_{\geqslant \lambda} \]
is fully faithful and dense. The induced topology on it is the atomic topology (every non-empty sieve is a cover), hence this induces an equivalence:
\[\Scott_{\kappa} Set^m_{\geqslant \lambda} \overset{\simeq}{\rightarrow} Sh( (Set^m_{< \kappa})^{op} , J_{at} ) \]

Finally if $\lambda \leqslant \lambda'$, the inclusion $i: Set^m_{\geqslant \lambda'} \rightarrow  Set^m_{\geqslant \lambda}$ is compatible this equivalence, i.e. on has commutative triangle of equivalences:
\[\begin{tikzcd}[ampersand replacement=\&]
\Scott_{\kappa} Set^m_{\geqslant \lambda'} \arrow{r}{\Scott_{\kappa} i }[swap]{\simeq} \arrow{dr}[swap]{\simeq} \& \Scott_{\kappa} Set^m_{\geqslant \lambda} \arrow{d}{\simeq} \\
 \& Sh( (Set^m_{< \kappa})^{op} , J_{at} ) 
\end{tikzcd}\]
\end{cor}

\begin{proof}

Applying \cref{lemma:morphismOutOfR_V} to morphisms from $R_V$ to $R_W$, and fixing some $\tau : V \rightarrow \lambda$:
\[ Hom(R_V, R_W) \overset{\simeq}{\rightarrow}  \{ x \in R_W(\lambda) | x \text{ has support in $\tau(V)$} \}\]
with the map being evaluation at $\tau$. An $x \in R_W (\lambda)$ is a monomorphisms $W \rightarrow \lambda$ and it has support in $\tau(V)$ if and only if its image is included in $\tau(V)$, so this shows that $Hom(R_V,R_W)$ is in bijection with monomorphisms from $W$ to $V$, with the bijection being simply given by the functoriality of $R_{\bullet}$, i.e. $R_{\bullet}$ is fully faithful.

The density of $R_{\bullet}$ is obtained by combining \cref{finitesupport} with \cref{lemma:morphismOutOfR_V}: for any $\Fcal \in \Scott_{\kappa}(Set^{m}_{\geqslant \lambda}$ any element $x \in \Fcal(X)$ is the image of a $x_0 \in \Fcal(\lambda)$ by some maps $\lambda \rightarrow X$ as $X$ is the $\kappa$-directed colimits of its subobject of cardinality $\kappa$. Then $x_0$ has $\kappa$-small support by \cref{finitesupport}, so one can construct a morphism $R_V \rightarrow \Fcal$ which has $x_0$ (and hence $x$) in its image. This shows that any object admits a covering (in the canonical topology of the topos) by the $R_V$.

Note that for any monomorphisms $V \hookrightarrow W$, the induced map $R_W \rightarrow R_V$ is an epimorphism: indeed any monomorphism $V \rightarrow X$ with $X$ of cardinality greater than $\lambda$ can be extended to $W$. So the induced topology on $(Set^m_{< \kappa})^{op}$ is the atomic topology (every non-empty sieve is a cover).

At this point Grothendieck comparison lemma implies that the functor
\[\Scott_{\kappa} Set^m_{\geqslant \lambda} \rightarrow Sh( (Set^m_{< \kappa})^{op} , J_{at} ) \]
sending any object $\Fcal$ to the (pre)sheaf $V \mapsto Hom(R_V, \Fcal)$ is an equivalence of categories. For $\lambda \leqslant \lambda'$ the functor:
\[ \Scott_{\kappa} Set^m_{\geqslant \lambda'} \overset{\Scott_{\kappa} i}{\rightarrow} \Scott_{\kappa} Set^m_{\geqslant \lambda} \]
sends the $R_V$ defined in  $\Scott_{\kappa} Set^m_{\geqslant \lambda'}$ to those defined in  $\Scott_{\kappa} Set^m_{\geqslant \lambda}$, and the description of morphisms between $R_V$ and another object $\Fcal$ given in \cref{lemma:morphismOutOfR_V} shows that the triangle of functor in the statement commutes and hence $\Scott_{\kappa} i $ is indeed an equivalence.
\end{proof}

\section{Generalizations}

\label{sec:generalizations}

The proofs above generalize easily to other categories than the category of sets, for examples to vector spaces (using dimension instead of cardinal) and algebraically closed fields (using transcendence degree). But it is quite unclear to me, what are the assumptions needed in general to make this proof works. They seem to involve both a good notion of dimension and some sort of Fraïssé theory available.

\begin{npar}
For example, I claim that all steps of the proof above generalizes to the following claims:

Let $K$ be any field and consider the category $A_{\lambda} = K-Vect^m_{\geqslant \lambda}$ of $K$-vector space of dimension at least $\lambda$ and linear monomorphism between them. Then following the same steps as the proof above allows to shows that:
\[ \Scott_{\kappa} A_{\lambda} \simeq Sh((K-Vect^m_{< \kappa})^{op},J_{At}) \]
where $K-Vect^m_{< \kappa}$ is the category of $K$-vector space of dimension smaller than $\kappa$. Similarly one deduce that the map:
\[ \Scott_{\kappa} A_{\lambda} \rightarrow \Scott_{\kappa} A_{\kappa}\]
is an equivalence. And similarly to \cref{Cor:MainFinitary} one concludes that the category $ K-Vect^m_{\geqslant \lambda}$ is not the category of $\kappa$-points of a $\kappa$-topos for any regular $\kappa< \lambda$. So is not the category of models of a $L_{\infty,\kappa}$-theory.

\end{npar}

\begin{npar}
Similarly, if $A$ denotes the category of algebraic closed fields of some fixed characteristic, and $A_{\lambda}$ the full subcategory of fields that have transcendence degree at least $\lambda$ over their prime subfield, then one also get isomorphisms for any $\lambda \geqslant \kappa$ :
\[  \Scott_{\kappa} A_{\lambda} \simeq Sh((ACF_{< \kappa})^{op},J_{At}) \]
Where $ACF_{< \kappa}$ is the full subcategory of algebraic closed fields of the same characteristic as above whose transcendence degree over their prime subfields is strictly smaller than $\kappa$. And as above, one deduce that the inclusion $A_{\lambda} \rightarrow A_{\kappa}$ induces an equivalence of categories:
\[ \Scott_{\kappa} A_{\lambda} \overset{\simeq}{\rightarrow} \Scott_{\kappa} A_{\kappa} \]
and that $A_{\lambda}$ cannot be the category of $\kappa$-points of a $\kappa$-topos nor the category of models of a $L_{\infty,\kappa}$-theory for any regular $\kappa < \lambda$.

\end{npar}

\bibliography{Biblio}{}
\bibliographystyle{plain}

\end{document}